\documentclass[preprint]{elsarticle}

\usepackage{amsmath}
\usepackage{amsthm}
\usepackage{amssymb}
\usepackage{latexsym}
\usepackage{textcomp}
\usepackage{mathtools}
\usepackage[T1]{fontenc}
\usepackage[sc]{mathpazo}
\usepackage{enumerate}
\usepackage{titlesec}
\usepackage[small,it]{caption}
\usepackage{pstricks}
\usepackage{graphicx, pst-plot, pst-node, pst-text, pst-tree, pstricks-add}
\usepackage{color}
\usepackage{calc}
\usepackage{gastex}
\definecolor{lightgray}{rgb}{0.9, 0.9, 0.9}
\definecolor{darkgray}{rgb}{0.7, 0.7, 0.7}
\definecolor{darkblue}{rgb}{0, 0, .4}

\usepackage[bookmarks]{hyperref}
\hypersetup{
    colorlinks=true,
    linkcolor=darkblue,
    anchorcolor=darkblue,
    citecolor=darkblue,
    urlcolor=darkblue,
    pdfpagemode=UseThumbs,
    pdftitle={The enumeration of a pattern class},
    pdfsubject={Combinatorics},
    pdfauthor={M. Albert and R. Brignall},
    pdfkeywords={todo}
}
\parskip\medskipamount
\theoremstyle{plain}
\newtheorem{theorem}{Theorem}[section]

\newtheorem{proposition}[theorem]{Proposition}

\theoremstyle{definition}

\theoremstyle{remark}

\newcounter{todocounter}

\makeatletter
\newfont{\footsc}{cmcsc10 at 8truept}
\newfont{\footbf}{cmbx10 at 8truept}
\newfont{\footrm}{cmr10 at 10truept}
\newcommand{\OEISlink}[1]{\href{http://oeis.org/#1}{#1}}
\newcommand{\OEISref}{\href{http://oeis.org/}{\emph{The On-Line Encyclopedia of Integer Sequences}}}
\newcommand{\OEIS}[1]{(Sequence \OEISlink{#1} in the \OEISref.)}

\newcommand{\Av}{\operatorname{Av}}

\newcommand{\A}{\mathcal{A}}

\newcommand{\C}{\mathcal{C}}

\renewcommand{\a}{a}
\renewcommand{\b}{b}
\renewcommand{\c}{c}
\renewcommand{\d}{d}

\newcounter{hone}
\newcounter{htwo}
\newcounter{vone}
\newcounter{vtwo}

\def\grid(#1,#2){%
\setcounter{hone}{#1*20}%
\setcounter{vone}{#2*20}%
\psset{xunit=0.005in, yunit=0.005in} \psset{linewidth=0.005in}%
\begin{pspicture}(0,0)(\thehone,\thevone)\drawlines(#1,#2)\end{pspicture}}

\def\drawlines(#1,#2){%
\setcounter{hone}{#1}%
\setcounter{vtwo}{20*\thehone}%
\addtocounter{hone}{1}%
\setcounter{vone}{#2}%
\setcounter{htwo}{20*\thevone}%
\addtocounter{vone}{1}%
\multido{\i=0+20}{\thehone}{\psline[linestyle=solid,linewidth=0.005in](\i,0)(\i,\thehtwo)}%
\multido{\i=0+20}{\thevone}{\psline[linestyle=solid,linewidth=0.005in](0,\i)(\thevtwo,\i)}}

\def\upline(#1,#2){%
\setcounter{hone}{(#1*20)-17}%
\setcounter{vone}{(#2*20)-17}%
\setcounter{htwo}{(#1*20)-3}%
\setcounter{vtwo}{(#2*20)-3}%
\psline[linestyle=solid,linewidth=0.02in](\thehone,\thevone)(\thehtwo,\thevtwo)}

\def\downline(#1,#2){%
\setcounter{hone}{(#1*20)-17}%
\setcounter{vtwo}{(#2*20)-17}%
\setcounter{htwo}{(#1*20)-3}%
\setcounter{vone}{(#2*20)-3}%
\psline[linestyle=solid,linewidth=0.02in](\thehone,\thevone)(\thehtwo,\thevtwo)}

\def\cellclass(#1,#2)#3{%
\setcounter{hone}{(#1*20)-10}%
\setcounter{vone}{(#2*20)-10}%
\rput[c](\thehone,\thevone){#3}}

\begin{document}

\begin{frontmatter}

\title{Enumerating indices of Schubert varieties defined by inclusions}

\author{Michael H.~Albert}
\address{Department of Computer Science, University of Otago, New Zealand}

\author{Robert Brignall}
\address{Department of Mathematics and Statistics, The Open University, UK}

\begin{abstract}
By extending the notion of grid classes to include infinite grids, we establish a structural characterisation of the simple permutations in Av(4231, 35142, 42513, 351624), a pattern class which has three different connections with algebraic geometry, including the specification of indices of Schubert varieties defined by inclusions. This characterisation leads to the enumeration of the class.
\end{abstract}

\end{frontmatter}

\section{Introduction}

Pattern classes, or permutation classes, are collections of permutations closed under taking subpermutations in a natural way (detailed definitions will be given in the next section). They have arisen in a number of different contexts within mathematics and theoretical computer science. One area where they appear is in the categorization of Schubert varieties having certain properties. A Schubert variety $X_w$ is determined by a permutation $w$, called its index. Certain properties of the variety are determined by patterns that the index $w$ avoids. For instance, Lakshmibai and Sandhya \cite{lakshmibai:criterion-for-smoothness:} showed that a Schubert variety is smooth if and only if its index contains neither 3412 nor 4231 as a subpermutation.

A larger class of Schubert varieties are those defined by inclusions. Gasharov and Reiner \cite{gasharov:cohomology:} proved that these are indexed by permutations not containing any of 4231, 35142, 42513, or 351624. As it turns out, the permutations that avoid these four patterns arise in two other unrelated contexts. First, Sj{\"o}strand \cite{sjostrand:bruhat:} established that this same avoidance condition characterizes permutations whose right hull\footnote{Roughly speaking, the \emph{right hull} of a permutation is the smallest right-aligned skew-Ferrers diagram that covers it.} covers exactly the elements below it in the Bruhat order. Second, Hultman, Linusson, Shareshian, and Sj{\"o}strand \cite{hultman:from-bruhat:} proved that the number of elements in the Bruhat order below a permutation is equal to the number of regions in the associated `inversion hyperplane arrangement' if and only if the permutation avoids these four patterns. In~\cite{hultman:from-bruhat:}, the authors observe that there seems to be no direct reason why the same avoidance conditions appear in three different settings, and indeed pose this as an open question. While we do not broach this question here, it is reasonable to expect that our results could provide a route forward by supplying a more concrete description of permutations that avoid these patterns.

For any given pattern class it is natural to ask what its enumeration is, i.e.~how many permutations of each length belong to it. More generally, one might hope to describe the structure of the permutations belonging to the class in some less abstract way than the negative criteria ``permutations not containing certain subpermutations''. In this paper we will carry out this program for the class mentioned above.

The enumeration of this class was raised as a problem at FPSAC 2008 and a (now known to be correct) conjecture concerning their enumeration was put by the first author based purely on numerical evidence. At the Permutation Patterns 2012 conference, Alexander Woo again raised the issue of the enumeration of this class, and we thank him for drawing our attention back to it. In recent years a variety of more general enumerative techniques for pattern classes, based on regular languages, \emph{simple permutations} and \emph{grid representations} have been developed. Though these do not apply directly to the problem at hand, we used and modified the underlying ideas to produce the structural characterisation and enumeration which we will demonstrate here.

\section{Terminology and methodology}

Throughout this paper permutations are written in one line notation i.e.~considered as sequences of positive integers, and are denoted by lower case Greek letters. The \emph{pattern} of a sequence of distinct positive integers is the unique permutation of the same length whose elements are in the same relative order. For instance the pattern of $47183$ is $34152$. We say that $\sigma$ is a \emph{subpermutation} of, or a \emph{pattern} in $\pi$ if $\pi$ has a subsequence with pattern $\sigma$ and in this case we write $\sigma \leq \pi$. If $\sigma$ is not a subpermutation of $\pi$ we say that $\pi$ \emph{avoids} $\sigma$. The number of elements in the domain of $\pi$ will be called its length and denoted $|\pi|$.


A \emph{pattern class}, $\C$, is any set of permutations with the property that if $\pi \in \C$ and $\sigma \leq \pi$, then $\sigma \in \C$. For any proper pattern class $\C$ we define its \emph{basis}, $X$, to be the minimal elements of the set of permutations  (with respect to the ordering $\leq$) not belonging to $\C$. Then $\pi \in \C$ if and only if $\pi$ avoids every permutation in $X$ and we write $\C = \Av(X)$. Conversely, given any antichain $X$ of permutations, $X$ is the basis of the class $\Av(X)$.

Throughout this paper we will be working with the class $\A = \Av(4231, 35142, 42513, 351624)$ (we omit superfluous braces). Since the basis of this class is fixed as a set after taking inverses, so is the class itself  as it is easy to check that $\sigma$ is a pattern in $\pi$ if and only if $\sigma^{-1}$ is a pattern in $\pi^{-1}$. This class also has a second proper symmetry: \emph{reverse complement}. The reverse complement $\pi^{rc}$ of a permutation $\pi$ is found by reversing the sequence that represents $\pi$ and then replacing $i$ by $|\pi|+1-i$ for each $i$. Again, it is easy to check that $\sigma \leq \pi$ if and only if $\sigma^{rc} \leq \pi^{rc}$ and the basis of $\A$ is fixed as a set by this operation. Because of these symmetries we can frequently reduce the number of cases that need to be considered in the arguments to follow.

A permutation is \emph{simple} if the only intervals that it maps to intervals are singletons and its entire domain. Equivalently, the corresponding sequence contains no factors whose elements form an interval except for singletons and the entire sequence. So for instance $479683152$ is not simple because of the  factor $7968$, while $2413$ is simple. Given a permutation $\pi$ of length $n$ and a sequence of permutations $\sigma_1, \sigma_2, \dots, \sigma_n$, we define the \emph{inflation}, $\pi[\sigma_1, \sigma_2, \dots, \sigma_n]$, to be the permutation of length $|\sigma_1| + |\sigma_2| + \dots + |\sigma_n|$ which is obtained from $\pi$ by replacing each element $\pi(i)$ by an interval of pattern $\sigma_i$ in such a way that the relative ordering between any two elements of the sequences introduced in place of $\pi(i)$ and $\pi(j)$ is the same as that between $\pi(i)$ and $\pi(j)$ themselves. For instance:
\[
2413[1, 21, 3142, 1] = 5 \, 87 \, 3142 \, 6.
\]

In general any permutation, $\pi$, is the inflation of a unique simple permutation called its \emph{skeleton}. If the skeleton has length greater than 2 then the permutations by which it is inflated are also uniquely determined by $\pi$. However, if the skeleton is 12 (in which case we say that $\pi$ is \emph{sum decomposable}) or $21$ (\emph{skew decomposable}) then the inflation is not unique -- for example $123 = 12[1,12] = 12[12,1]$. Uniqueness is generally enforced in this case by insisting that the first of the permutations in the inflation be indecomposable (i.e.~we would favour the first of the two alternatives for 123).

As a general strategy in trying to enumerate a pattern class we can first attempt to enumerate its simple permutations, and then describe how these can be inflated (often handling sum and skew decompositions as a special case). In some vague sense the simple permutations of a class capture its essential character and are frequently far less numerous and more easily described than the class as a whole. For example, letting $s_n$ denote the number of simple permutations of length $n$ in $\A$ we have:
\begin{center}
\begin{tabular}{c|cccccc}
$n$ & 4 & 5 & 6 & 7 & 8 & 9 \\ \hline
$s_n$ & 2 & 4 & 14 & 40 & 122 & 364
\end{tabular}
\end{center}
From this sequence it is easy to form the conjecture that $s_{n+1} = 3 s_n + 2(-1)^{n+1}$. Such a straightforward formula suggests that the strategy specified above has good chance of success.

The \emph{extreme pattern} of a permutation is the pattern defined by its first, last, greatest and least elements. A simple permutation of length greater than 2 can neither begin nor end with its greatest or least elements (otherwise the remaining elements would form a proper, non-singleton, interval) and so its extreme pattern must be one of 2143, 2413, 3142, 3412.

In the next section we give a structural characterisation of the simple permutations in $\A$. In turn this leads to their description by means of a regular language, the description of their inflations, and finally the generating function for $\A$ itself using the standard toolkit of symbolic combinatorics.

\section{Characterising simple permutations in $\A$}

\begin{proposition}
No simple permutation in $\A$ has extreme pattern $3412$.
\end{proposition}

\begin{proof}
Suppose that $\pi \in \A$ were simple and had extreme pattern 3412, specifically $cdab$. This situation is illustrated below. In this and all subsequent illustrations the grid as a whole represents the area in which the graph of the permutation can be drawn, some of its elements will be represented by dots and possibly named, and a shaded cell contains no elements, either because such an element would create a forbidden pattern (dark grey), or because we have made specific choices to ensure that this is so (light grey). \\
\centerline{
\psset{xunit=2pt, yunit=2pt, runit=1.5pt}
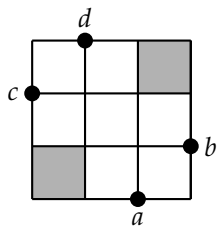
\begin{pspicture}(0,0)(50,50)
\pspolygon*[linecolor=darkgray](10,10)(20,10)(20,20)(10,20)
\pspolygon*[linecolor=darkgray](30,30)(40,30)(40,40)(30,40)
\pscircle*[linecolor=black](10,30){2.0}
\uput[180](10,30){$c$}
\pscircle*[linecolor=black](20,40){2.0}
\uput[90](20,40){$d$}
\pscircle*[linecolor=black](30,10){2.0}
\uput[-90](30,10){$a$}
\pscircle*[linecolor=black](40,20){2.0}
\uput[0](40,20){$b$}
\psline(10,10)(10,40)
\psline(10,10)(40,10)
\psline(20,10)(20,40)
\psline(10,20)(40,20)
\psline(30,10)(30,40)
\psline(10,30)(40,30)
\psline(40,10)(40,40)
\psline(10,40)(40,40)
\psline(40,10)(40,40)
\end{pspicture}
}
Remembering that $\pi$ is supposed to be simple, either there must be some element lying between $c$ and $d$ which is smaller than $c$ or some later element whose value lies between $c$ and $d$, or both. We say that ``the interval by position and value between $c$ and $d$ must be split''. More generally in a simple permutation we can take the rectangular region bounded by two or more specified elements and demand that it be split by some new specified element if it has not already been split by some other specified element.

To split the rectangular region defined by the elements $c$ and $d$ requires the presence of an element $x$ either in the square below it, or to its right. The two cases are similar (in fact symmetric under the symmetry ``inverse followed by reverse complement'') and we consider only the first. Take as $x$ the least such element. \\
\centerline{
\psset{xunit=2pt, yunit=2pt, runit=1.5pt}
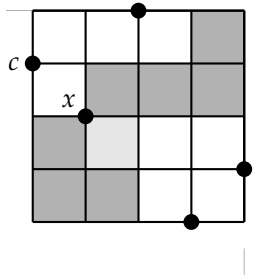
\begin{pspicture}(0,0)(60,60)
\pspolygon*[linecolor=darkgray](0,50)(10,50)(10,50)(0,50)
\pspolygon*[linecolor=darkgray](10,10)(20,10)(20,20)(10,20)
\pspolygon*[linecolor=darkgray](10,20)(20,20)(20,30)(10,30)
\pspolygon*[linecolor=darkgray](10,50)(20,50)(20,50)(10,50)
\pspolygon*[linecolor=darkgray](20,10)(30,10)(30,20)(20,20)
\pspolygon*[linecolor=lightgray](20,20)(30,20)(30,30)(20,30)
\pspolygon*[linecolor=darkgray](20,30)(30,30)(30,40)(20,40)
\pspolygon*[linecolor=darkgray](30,30)(40,30)(40,40)(30,40)
\pspolygon*[linecolor=darkgray](40,30)(50,30)(50,40)(40,40)
\pspolygon*[linecolor=darkgray](40,40)(50,40)(50,50)(40,50)
\pspolygon*[linecolor=darkgray](50,0)(50,0)(50,10)(50,10)
\pspolygon*[linecolor=darkgray](50,30)(50,30)(50,40)(50,40)
\pscircle*[linecolor=black](10,40){2.0}
\uput[180](10,40){$c$}
\pscircle*[linecolor=black](20,30){2.0}
\uput[135](20,30){$x$}
\pscircle*[linecolor=black](30,50){2.0}
\pscircle*[linecolor=black](40,10){2.0}
\pscircle*[linecolor=black](50,20){2.0}
\psline(10,10)(10,50)
\psline(10,10)(50,10)
\psline(20,10)(20,50)
\psline(10,20)(50,20)
\psline(30,10)(30,50)
\psline(10,30)(50,30)
\psline(40,10)(40,50)
\psline(10,40)(50,40)
\psline(50,10)(50,50)
\psline(10,50)(50,50)
\psline(50,10)(50,50)
\psline(10,50)(50,50)
\end{pspicture}
}
Splitting the box defined by $\{c,x\}$ requires another element $y$ in the top left square. Take the largest such element. \\
\centerline{
\psset{xunit=2pt, yunit=2pt, runit=2pt}
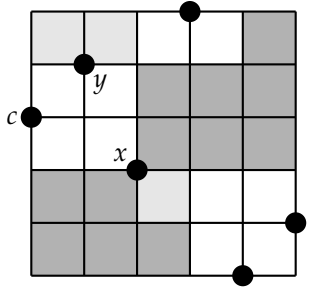
\begin{pspicture}(0,0)(70,70)
\pspolygon*[linecolor=darkgray](10,10)(20,10)(20,20)(10,20)
\pspolygon*[linecolor=darkgray](10,20)(20,20)(20,30)(10,30)
\pspolygon*[linecolor=lightgray](10,50)(20,50)(20,60)(10,60)
\pspolygon*[linecolor=darkgray](20,10)(30,10)(30,20)(20,20)
\pspolygon*[linecolor=darkgray](20,20)(30,20)(30,30)(20,30)
\pspolygon*[linecolor=lightgray](20,50)(30,50)(30,60)(20,60)
\pspolygon*[linecolor=darkgray](30,10)(40,10)(40,20)(30,20)
\pspolygon*[linecolor=lightgray](30,20)(40,20)(40,30)(30,30)
\pspolygon*[linecolor=darkgray](30,30)(40,30)(40,40)(30,40)
\pspolygon*[linecolor=darkgray](30,40)(40,40)(40,50)(30,50)
\pspolygon*[linecolor=darkgray](40,30)(50,30)(50,40)(40,40)
\pspolygon*[linecolor=darkgray](40,40)(50,40)(50,50)(40,50)
\pspolygon*[linecolor=darkgray](50,30)(60,30)(60,40)(50,40)
\pspolygon*[linecolor=darkgray](50,40)(60,40)(60,50)(50,50)
\pspolygon*[linecolor=darkgray](50,50)(60,50)(60,60)(50,60)
\pscircle*[linecolor=black](10,40){2.0}
\uput[180](10,40){$c$}
\pscircle*[linecolor=black](20,50){2.0}
\uput[-45](20,50){$y$}
\pscircle*[linecolor=black](30,30){2.0}
\uput[135](30,30){$x$}
\pscircle*[linecolor=black](40,60){2.0}
\pscircle*[linecolor=black](50,10){2.0}
\pscircle*[linecolor=black](60,20){2.0}
\psline(10,	10)(10,60)
\psline(10,10)(60,10)
\psline(20,10)(20,60)
\psline(10,20)(60,20)
\psline(30,10)(30,60)
\psline(10,30)(60,30)
\psline(40,10)(40,60)
\psline(10,40)(60,40)
\psline(50,10)(50,60)
\psline(10,50)(60,50)
\psline(60,10)(60,60)
\psline(10,60)(60,60)
\end{pspicture}
}
Now the box defined by $\{c,x,y\}$ defines an interval which is unsplit and unsplittable, and hence a contradiction that establishes the result.
\end{proof}

The technique used above, either to establish that some configuration is impossible or to impose certain structure within a simple permutation of $\A$ will be used throughout and henceforth we shall not give those arguments in such great detail.

\begin{proposition}
\label{Prop_2413}
If the extreme points of a simple permutation $\pi$ in $\A$ have the pattern 2413, say as $bdac$, then the permutation is $N$-shaped, consisting of an increasing segment from $b$ to $d$, a decreasing one from $d$ to $a$ and an increasing one from $a$ to $c$.
\end{proposition}

\begin{proof}
First we establish that there is no element $x$ below $b$ and to the left of $d$. Otherwise, we could take the least such. The potential interval bounded by $\{b,x\}$ can only be split between $b$ and $x$ and above $b$. Let $y$ be the largest such split. Then the box bounded by $\{b,x,y\}$ is an interval.

So all elements lying between $b$ and $d$ lie above $b$. Suppose that this set of elements is not monotone increasing. Choose an inversion $yx$ in this interval with greatest possible top and least possible bottom (for that top). Now the box defined by $\{x,y\}$ can only be split to its left. Taking the leftmost such split, $z$ yields the picture below, in which the box bounded by $\{x,y,z\}$ is an interval.

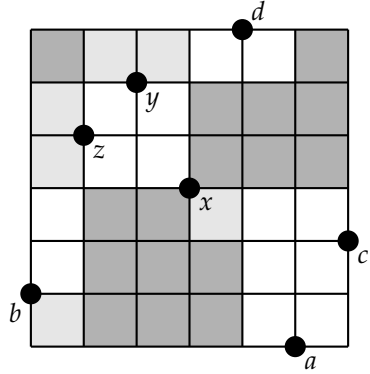
\centerline{
\psset{xunit=2pt, yunit=2pt, runit=2pt}
\begin{pspicture}(0,0)(80,80)
\pspolygon*[linecolor=lightgray](10,10)(20,10)(20,20)(10,20)
\pspolygon*[linecolor=lightgray](10,40)(20,40)(20,50)(10,50)
\pspolygon*[linecolor=lightgray](10,50)(20,50)(20,60)(10,60)
\pspolygon*[linecolor=darkgray](10,60)(20,60)(20,70)(10,70)
\pspolygon*[linecolor=darkgray](20,10)(30,10)(30,20)(20,20)
\pspolygon*[linecolor=darkgray](20,20)(30,20)(30,30)(20,30)
\pspolygon*[linecolor=darkgray](20,30)(30,30)(30,40)(20,40)
\pspolygon*[linecolor=lightgray](20,60)(30,60)(30,70)(20,70)
\pspolygon*[linecolor=darkgray](30,10)(40,10)(40,20)(30,20)
\pspolygon*[linecolor=darkgray](30,20)(40,20)(40,30)(30,30)
\pspolygon*[linecolor=darkgray](30,30)(40,30)(40,40)(30,40)
\pspolygon*[linecolor=lightgray](30,60)(40,60)(40,70)(30,70)
\pspolygon*[linecolor=darkgray](40,10)(50,10)(50,20)(40,20)
\pspolygon*[linecolor=darkgray](40,20)(50,20)(50,30)(40,30)
\pspolygon*[linecolor=lightgray](40,30)(50,30)(50,40)(40,40)
\pspolygon*[linecolor=darkgray](40,40)(50,40)(50,50)(40,50)
\pspolygon*[linecolor=darkgray](40,50)(50,50)(50,60)(40,60)
\pspolygon*[linecolor=darkgray](50,40)(60,40)(60,50)(50,50)
\pspolygon*[linecolor=darkgray](50,50)(60,50)(60,60)(50,60)
\pspolygon*[linecolor=darkgray](60,40)(70,40)(70,50)(60,50)
\pspolygon*[linecolor=darkgray](60,50)(70,50)(70,60)(60,60)
\pspolygon*[linecolor=darkgray](60,60)(70,60)(70,70)(60,70)
\pscircle*[linecolor=black](10,20){2.0}
\uput[-135](10,20){$b$}
\pscircle*[linecolor=black](20,50){2.0}
\uput[-45](20,50){$z$}
\pscircle*[linecolor=black](30,60){2.0}
\uput[-45](30,60){$y$}
\pscircle*[linecolor=black](40,40){2.0}
\uput[-45](40,40){$x$}
\pscircle*[linecolor=black](50,70){2.0}
\uput[45](50,70){$d$}
\pscircle*[linecolor=black](60,10){2.0}
\uput[-45](60,10){$a$}
\pscircle*[linecolor=black](70,30){2.0}
\uput[-45](70,30){$c$}
\psline(10,10)(10,70)
\psline(10,10)(70,10)
\psline(20,10)(20,70)
\psline(10,20)(70,20)
\psline(30,10)(30,70)
\psline(10,30)(70,30)
\psline(40,10)(40,70)
\psline(10,40)(70,40)
\psline(50,10)(50,70)
\psline(10,50)(70,50)
\psline(60,10)(60,70)
\psline(10,60)(70,60)
\psline(70,10)(70,70)
\psline(10,70)(70,70)
\end{pspicture}
}

The segment from $d$ to $a$ is decreasing by 4231-avoidance, and that from $a$ to $c$ is increasing by the reverse complement symmetry of the argument used in the first part.
\end{proof}

Using the inverse symmetry for $\A$ we also obtain:

\begin{proposition}
If the extreme points of a simple permutation $\pi$ in $\A$ have the pattern 3142, say as $cadb$ then the permutation is $S$ shaped: the values between $a$ and $b$ forming an increasing sequence, those between $b$ and $c$ a decreasing one, and those between $c$ and $d$ an increasing one.
\end{proposition}

Those two results were really just warm-ups for the general description of the simple permutations in $\A$. We provide part of that description now. Let $\pi$ be a simple permutation of $\A$ (of length at least four) and suppose that 1 does not occur in the second position of $\pi$. Using the inverse symmetry, this can always be assumed: if both $\pi$ and $\pi^{-1}$ had 1 in the second position, then $\pi$ would begin with 21 and hence could not be simple. Now write:
\[
\pi = b \cdots d \cdots a \cdots e \cdots c \cdots
\]
where: $a = 1$, $b$ is the first element of $\pi$, $d$ is the greatest element between $b$ and $a$, $c$ is the rightmost element whose value is less than $d$, and $e$ is the second rightmost element whose value is less than $d$.

Our first claim is that $b \neq d$. Suppose it were the case that $b = d$. Then there are no elements larger than $b$ lying between $b$ and $a$, but by assumption there is some element $x$ between $b$ and $a$. Take the least such. Splitting the box determined by $\{a,x\}$ requires an element $y$ beyond $a$ and between $a$ and $x$ in value. Take the rightmost such. Now in order to split $\{a,x,y\}$ we require an element $z$ between $a$ and $y$ and larger than $x$. In fact $z$ is larger than $b$, else $bxzy \sim 4231$. The pattern of the five elements we have considered until now is $bxazy \sim 43152$. The box formed by $\{b,x\}$ is not split to the left ($b$ is leftmost), above (by assumption), below (by 4231 avoidance), or to the right (by 4231 and 42513 avoidance), so we have a contradiction to simplicity.

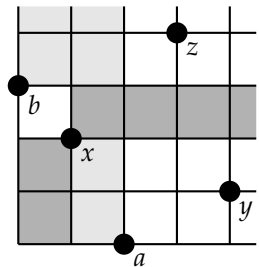
\centerline{
\psset{xunit=2pt, yunit=2pt, runit=2pt}
\begin{pspicture}(0,0)(70,70)
\pspolygon*[linecolor=darkgray](10,10)(20,10)(20,20)(10,20)
\pspolygon*[linecolor=darkgray](10,20)(20,20)(20,30)(10,30)
\pspolygon*[linecolor=lightgray](10,40)(20,40)(20,50)(10,50)
\pspolygon*[linecolor=lightgray](10,50)(30,50)(30,55)(10,55)
\pspolygon*[linecolor=lightgray](20,10)(30,10)(30,20)(20,20)
\pspolygon*[linecolor=lightgray](20,20)(30,20)(30,30)(20,30)
\pspolygon*[linecolor=darkgray](20,30)(30,30)(30,40)(20,40)
\pspolygon*[linecolor=lightgray](20,40)(30,40)(30,50)(20,50)
\pspolygon*[linecolor=darkgray](30,30)(40,30)(40,40)(30,40)
\pspolygon*[linecolor=darkgray](40,30)(50,30)(50,40)(40,40)
\pspolygon*[linecolor=darkgray](50,30)(55,30)(55,40)(50,40)
\pscircle*[linecolor=black](10,40){2.0}
\uput[-45](10,40){$b$}
\pscircle*[linecolor=black](20,30){2.0}
\uput[-45](20,30){$x$}
\pscircle*[linecolor=black](30,10){2.0}
\uput[-45](30,10){$a$}
\pscircle*[linecolor=black](40,50){2.0}
\uput[-45](40,50){$z$}
\pscircle*[linecolor=black](50,20){2.0}
\uput[-45](50,20){$y$}
\psline(10,10)(10,55)
\psline(10,10)(55,10)
\psline(20,10)(20,55)
\psline(10,20)(55,20)
\psline(30,10)(30,55)
\psline(10,30)(55,30)
\psline(40,10)(40,55)
\psline(10,40)(55,40)
\psline(50,10)(50,55)
\psline(10,50)(55,50)
\end{pspicture}
}

Our second claim is that the elements between the positions of $b$ and $d$ form an increasing sequence. This is proved exactly as in Proposition \ref{Prop_2413} above. We also claim that all these elements, except for $d$, are less than $c$. Otherwise take the leftmost such, $x$. The box $\{d,x\}$ cannot be split below (due to the increasing sequence from $b$ to $d$) so is split to the right. But, as usual, the rightmost such split $y$ gives an interval in the box $\{d,x,y\}$.

Similarly, the elements between $d$ and $a$ in position form a decreasing sequence and all lie below $c$. And again similarly, the elements between $a$ and $c$ in position whose values lie below $c$ form an increasing sequence. The elements between $a$ and $c$ in position which lie below $d$ and above $c$ form a decreasing sequence (from 4231 avoidance alone). Moreover, none of these elements can lie to the left of any element $x$, between $a$ and $c$ of value less than $c$, and also less than the value of some $y$ (except $d$) lying between $b$ and $a$: such elements would give either a 4231 if $y$ occurred after $d$, or a 35142 if it occurred before.

There are no elements between $e$ and $c$ whose value lies below $d$, so the box $\{c,e\}$ must be split by some element $x$ larger than $d$. Taking the least such $x$ it follows from avoidance conditions that all elements between $a$ and $e$ lie below $x$. The usual argument implies that these elements form an increasing sequence. At this point we have identified a significant part of $\pi$ as shown in Figure \ref{FIG_FirstPart}.

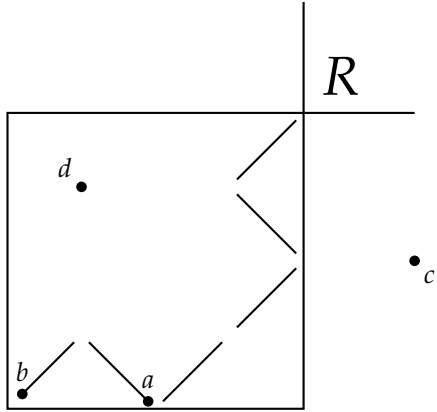
\begin{figure}
\centerline{
\psset{xunit=1.4pt, yunit=1.4pt, runit=1pt}
\begin{pspicture}(0,0)(120,120)
\psline(4,4)(18,18)
\psline(22,18)(38,2)
\psline(42,2)(58,18)
\psline(62,22)(78,38)
\psline(78,42)(62,58)
\psline(62,62)(78,78)
\pscircle*[linecolor=black](4,4){2.0}
\uput[90](4,4){$b$}
\pscircle*[linecolor=black](38,2){2.0}
\uput[90](38,2){$a$}
\pscircle*[linecolor=black](20,60){2.0}
\uput[135](20,60){$d$}
\pscircle*[linecolor=black](110,40){2.0}
\uput[-45](110,40){$c$}
\pspolygon(0,0)(80,0)(80,80)(0,80)
\psline(80,110)(80,80)(110,80)
\rput(90,90){\huge $R$}
\end{pspicture}
}
\caption{The first part of the form of a simple permutation in $\A$ with extreme pattern $2143$.}
\label{FIG_FirstPart}
\end{figure}

We refer now to Figure \ref{FIG_FirstPart}. In this figure, the outlined region in the lower left will be referred to as ``the lower box''. The element $e$ might be either just below or just above $c$. It is split from $c$ by either the uppermost element of the lower box, or (at least) the first element in the area marked $R$. The increasing sequence in the upper right corner of the lower box consists of all the remaining elements that lie to the left of $e$, and possibly the element immediately to the right of $e$ if it should happen to be the least element to the right of $e$ other than $c$. The remainder of $\pi$ lies entirely within the area marked $R$.

Still with reference to Figure \ref{FIG_FirstPart} we now remove the entire section of $\pi$ in the lower box, except for the point $d$. We claim that the remaining permutation $\pi'$ is still simple. For suppose that it had a proper interval, $I$. If this interval included $c$ then it would also include $d$. Then it, and all the elements we removed, would be a proper interval of $\pi$. If $I$ included $d$ but not $c$ then it would include the leftmost element in the region marked $R$. But $R$ cannot begin with its least element (otherwise, it would either have been added to the lower box, or we would already have added this element -- then the element we added and the first element of $R$ would have been an interval in $\pi$). So $I$ contains at least two points of $R$ and $I \setminus \{d\}$ is an interval of $\pi$. Finally, if it included neither $c$ nor $d$ then it is an interval of $\pi$. Since each case leads to a contradiction, we cannot have a proper interval in $\pi'$.

This decomposition allows us to continue the description of the simple permutations in $\A$. Either what is left over is covered by one of the earlier propositions (specifically, the extreme points would have to have pattern 2413), or we can construct another box like the one we have just built (with $d$ and $c$ taking on the roles of $b$ and $a$ respectively). This establishes:

\begin{figure}
\centerline{
\psset{xunit=1.5pt, yunit=1.5pt, runit=2pt}
\begin{pspicture}(-20,-30)(180,160)
\def\singleBox{%
\psline(0,0)(9.5,9.5)
\psline(10.5,9.5)(19.5,0.5)
\psline(20.5,0.5)(29.5,9.5)
\psline(30.5,10.5)(39.5,19.5)
\psline(39.5,20.5)(30.5,29.5)
\psline(30.5,30.5)(39.5,39.5)
\pscircle*(10,30){1}
\pscircle*(60,20){1}
}
\multips{0}(0,0)(40,40){4}{\singleBox}
\pscircle*(-30,-10){1}
\pscircle*(20,-20){1}
\uput[-135](57,48){\small $A$}
\uput[-135](77,68){\small $B$}
\uput[-45](-30,-10){\footnotesize 1}
\uput[-45](20,-20){\footnotesize 2}
\uput[-45](10,30){\footnotesize 3}
\uput[-45](60,20){\footnotesize 4}
\uput[-45](50,80){\footnotesize 5}
\uput[-45](100,60){\footnotesize 6}
\uput[-45](90,120){\footnotesize 7}
\uput[-45](140,100){\footnotesize 8}
\uput[-45](130,160){\footnotesize 9}
\uput[-45](180,140){\footnotesize 10}
\pspolygon*[linecolor=white](150,130)(160,130)(160,160)(150,160)
\end{pspicture}
}
\caption{The form of a simple permutation with extreme pattern 2143 in $\A$. Simple permutations whose extreme pattern is 2413 can also be interpreted as a special case of this form. The ``isolated'' (numbered) points must be present, and define the shape of the rest of the permutation, which we will call the \emph{crenellation}. The crenellation can consist of arbitrarily many line segments, and each segment can contain any number of elements, but these must be placed in such a way as to avoid creating intervals. The labels $A$ and $B$ mark the two different types of \emph{sides} of the crenellation and are referred to in the proof of Theorem \ref{THM_InClass}. The numbers on the isolated points indicate the order in which they are encoded in the representation of the simple permutations using a regular language.}
\label{FIG_Crenellation}
\end{figure}
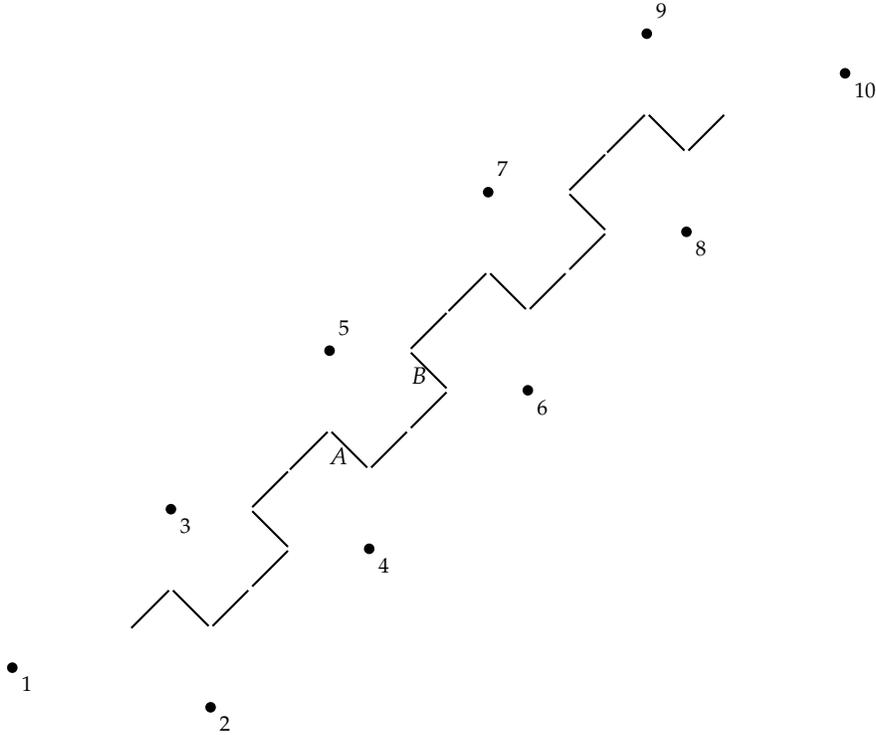

\begin{theorem}
\label{THM_InClass}
Let $\pi \in \A$ be simple and suppose that $\pi(2) \neq 1$. Then $\pi$ has a decomposition of the form described in the caption of Figure \ref{FIG_Crenellation} and illustrated there for a particular case. Moreover, every simple permutation of this form lies in $\A$.
\end{theorem}

\begin{proof}
The first part of the result is established in the previous paragraphs. To establish the second part we need to show that every permutation, $\pi$, of this form belongs to $\A$, i.e.~avoids all of 4231, 35142, 42513 and 351624.

Owing to its suggestive shape, we will refer to the set of points in a simple permutation other than the isolated points in Figure~\ref{FIG_Crenellation} as the \emph{crenellation}. We further subdivide the crenellation into the \emph{top} and \emph{bottom} (being the double-length line segments), and the \emph{sides} (the shorter line segments).

Note that the top row of dots, and the top of each crenellation are left to right maxima, while the bottom row of dots and the bottom of each crenellation are right to left minima. In a purported 4231 pattern therefore, none of these can play the role of 2 or 3. So that role must be played by two elements from the sides of a crenellation, and from different sides since an individual side is decreasing. But, two such elements have either no larger element before both of them, or no smaller element after.

Similarly, in a purported 35142 pattern, the 4 must come from a side of a crenellation. But, if it comes from a side of type $A$ there is no possibility for the 514 part of the pattern, since all the larger elements before it come after all the smaller elements before it. Similarly, if it comes from a side of type $B$ there is no possibility for $342$ as all the elements smaller than it and after it are larger than the elements smaller than it and before it. By reverse-complement symmetry, the pattern 42513 cannot occur either.

Finally, suppose there were a 351624 pattern. First we check that the 3 could come from the side of a crenellation. It is immediate that it could not be of type $A$ as the following smaller elements have no larger elements between them. Suppose that it came from a side of type $B$. In order to leave room for the 1, the 5 would have to come from the first half of the immediately following crenellation top -- but then there is no possible 4. A similar case by case analysis establishes that all of 3, 5 and 1 (and hence by reverse-complement symmetry also 6, 2 and 4) must come from the isolated elements and the top and bottoms of the crenellations. Consider now the 3. Since this is a left to right maximum it must come from an isolated point in the top group or the top of a crenellation. The latter is impossible, because all the smaller elements (that remain) form an interval by position, so we cannot achieve the 162 part. On the other hand, taking one of the top isolated points will force us to take 56 in the first half of the next top segment (in order to bracket the 1 and be followed by the 2), and this prevents the addition of a 4.
\end{proof}

\section{Enumerating $\A$}

In order to enumerate $\A$ we use the result of the previous section to first enumerate the simple permutations of $\A$. Note that by Theorem~\ref{THM_InClass} any simple $\pi\in \A$ falls into one of two classes: either the leftmost point of $\pi$ is 2, or $\pi(2)=1$. By inverse symmetry these two classes are equinumerous, so we will enumerate exactly half of the simples in $\A$ by considering only those for which $\pi(2)\neq 1$.

First, we need to refine the viewpoint shown in Figure \ref{FIG_Crenellation} somewhat. The reason is that, while every simple permutation $\pi\in\A$ with $\pi(2)\neq 1$ is of this form, and every permutation of this form lies in $\A$ (as established by Theorem \ref{THM_InClass}), it is not the case that every permutation of this form is simple.

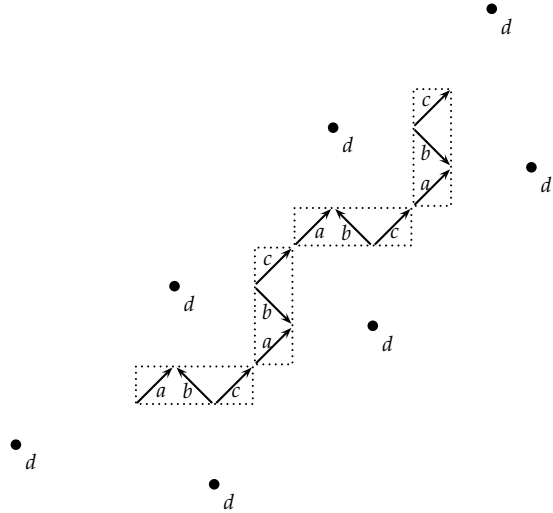
\begin{figure}
\centerline{
\psset{xunit=1.5pt, yunit=1.5pt, runit=2pt}
\begin{pspicture}(-20,-30)(100,100)
\def\singleBox{%
\psline{->}(0.5,0.5)(9.5,9.5)
\psline{<-}(10.5,9.5)(19.5,0.5)
\psline{->}(20.5,0.5)(29.5,9.5)
\psline{->}(30.5,10.5)(39.5,19.5)
\psline{<-}(39.5,20.5)(30.5,29.5)
\psline{->}(30.5,30.5)(39.5,39.5)
\psframe[linestyle=dotted,dotsep=1.5pt,fillstyle=none](0,0)(30,10)
\psframe[linestyle=dotted,dotsep=1.5pt,fillstyle=none](30,10)(40,40)
\pscircle*(10,30){1}
\pscircle*(60,20){1}%
}
\multips{0}(0,0)(40,40){2}{\singleBox}
\pscircle*(-30,-10){1}
\pscircle*(20,-20){1}
\pscircle*(90,100){1}
\uput[-45](3,7){\footnotesize $a$}
\uput[-135](17,8){\footnotesize $b$}
\uput[-45](22,7){\footnotesize $c$}
\uput[135](37,12){\footnotesize $a$}
\uput[-135](37,28){\footnotesize $b$}
\uput[135](37,33){\footnotesize $c$}
\uput[-45](43,47){\footnotesize $a$}
\uput[-135](57,48){\footnotesize $b$}
\uput[-45](62,47){\footnotesize $c$}
\uput[135](77,52){\footnotesize $a$}
\uput[-135](77,68){\footnotesize $b$}
\uput[135](77,73){\footnotesize $c$}
\uput[-45](-30,-10){\footnotesize $d$}
\uput[-45](20,-20){\footnotesize $d$}
\uput[-45](10,30){\footnotesize $d$}
\uput[-45](60,20){\footnotesize $d$}
\uput[-45](50,70){\footnotesize $d$}
\uput[-45](100,60){\footnotesize $d$}
\uput[-45](90,100){\footnotesize $d$}
\pspolygon*[linecolor=white](150,130)(160,130)(160,160)(150,160)
\end{pspicture}
}
\caption{The encoding used to describe the simple permutations of $\A$. The three line segments within each dotted box correspond to a block encoded by a contiguous factor in $\{\a,\b,\c\}$.}
\label{FIG_Encoding}
\end{figure}

We seek an encoding into a language over a finite alphabet in which the simple permutations of $\A$ form a regular language. By standard techniques we will then enumerate this language (using separate variables for each letter), and then the allowed inflations of each point will provide an enumeration of $\A$.

The encoding uses a four letter alphabet, $\Sigma = \{\a, \b, \c, \d \}$. The letter $\d$ will be used to encode each of the isolated points, in the order illustrated in Figure~\ref{FIG_Crenellation}. We divide the crenellation into \emph{blocks} of three line segments at a time, as shown in Figure~\ref{FIG_Encoding}. There are two types of block: $N$-shaped, where the three line segments interact by value, and $S$-shaped, where they interact by position. Note that no two distinct blocks interact either by value or position; the only interactions that a given block participates in are with itself and exactly two isolated points. In the case of $N$-shaped blocks one of these isolated points lies below and one lies above the block, while for $S$-shaped blocks the isolated elements lie to the left and to the right. Additionally, note that all of these interacting isolated points must be present (otherwise the permutation being encoded is sum decomposable), and that it is possible for a block to contain no points.

We encode each block using the three letters $\a$, $\b$ and $\c$, with $N$-shaped blocks being encoded from bottom to top, and $S$-shaped blocks encoded from left to right, as illustrated in Figure~\ref{FIG_Encoding}. For a word $w$ which encodes a given simple permutation in the class, each block is encoded by a contiguous factor of $w$ lying between two consecutive occurences of $d$ in $w$. So, within an $N$-shaped block for instance, the first letter of this factor indicates the type of the least element of the block, the second letter the type of the second least element, \dots, and the last letter the type of the greatest element of the block. Immediately preceding this contiguous factor is the letter $\d$ corresponding to the isolated point below the block (in the case of $N$-shaped) or (for $S$-shaped) to the left of the block, and the factor is immediately followed by the letter $\d$ encoding the isolated point above or to the right. Thus, each successive $\d$ triggers a transition from one block to the next (alternating between $N$- and $S$-shaped blocks).

By our choice that $\pi(2)\neq 1$, the first block to be encoded is $N$-shaped and is preceded by two points, both of which we encode by $\d$. This ensures uniqueness of the encoding, as the leftmost point could equivalently be encoded by $\a$ as a member of the first $N$-shaped block. Thus, the word begins $\d\d$, and similarly it must also end $\d\d$. Note, however, that the final block to be encoded can be either $N$- or $S$-shaped, and this is determined by the number of $d$s occurring in the word.

Finally, we claim that the following conditions are both necessary and sufficient to guarantee that a word, $w$, encodes a simple permutation, and that this encoding is unique:
\begin{itemize}
\item
$w$ must begin and end $\d\d$,
\item
$w$ must contain no $\a\a$, $\b\b$ or $\c\c$ factors,
\item
$w$ cannot begin $\d\d\a$ or end $\c\d\d$,
\item
a block cannot begin with $\a$ (i.e.\ the factor $\d\a$ is forbidden).
\end{itemize}

To justify this claim, first observe that, since all elements encoded by $\d$ must be present, there are only two possible ambiguities: that the first and last points may be encoded by $\a$ or $\d$ and $\c$ or $\d$ respectively, and that the first point of a block whose encoding begins with $\a$ could equivalently be encoded by $\c$ in the preceding block. Both of these ambiguities are excluded in the above conditions, and therefore each permutation is encoded uniquely.

We now show that all permutations encoded in this way are simple. Suppose, for a contradiction, that $\pi$ is encoded by a word of the form specified above, but that $\pi$ contains a proper non-singleton interval $I$. First, $I$ cannot contain more than one point encoded by $\d$ as otherwise $I$ must contain all points encoded by $\d$, and consequently $I$ must contain all points of $\pi$.

Next, if $I$ contains no points encoded by $d$, then $I$ consists entirely of points within one line segment of the crenellation: since the repeated factors $\a\a$, $\b\b$ and $\c\c$ are forbidden, the only possibility is that $I$ consists of the last point (encoded by $c$) from one block, and the first point (encoded by $a$) from the next. This, however, is impossible, since no block can begin with the letter $\a$.

Finally, suppose $I$ contains exactly one point encoded by $\d$. Unless this point is the first or the last point of the encoding, then $I$ cannot contain any other points, otherwise it would necessarily contain another point encoded by $\d$. If $I$ contains the first point of $\pi$, then the remaining points of $I$ must come from the first line segment of the first block, all of which are encoded by $\a$. However, this first block cannot begin with the letter $\a$, and so $I$ contains no such points. Similarly, if $I$ contains the final encoded point, then the only other points that can be contained in $I$ are encoded by $\c$ in the final block, but the final block is not allowed to finish with the letter $\c$, so this too is impossible.

Thus, the simple permutations in $\A$ with leftmost element equal to 2 are encoded by a regular language. Because the rules that specify this language are so uncomplicated it is easy to describe an automaton that recognises it.  One such is described by the transition table below, with initial state $I$, and unique accepting state $DD$ (the state names reflect a suffix of the sequence processed to reach them, which is sufficiently long to ensure that none of the conditions are violated):\\
\centerline{
\begin{tabular}{c|cccc}
&$\a$&$\b$&$\c$&$\d$\\\hline
 $I$&&$B$&$C$&$D$\\
$A$&&$B$&$C$&$D$\\
$B$&$A$&&$C$&$D$\\
$C$&$A$&$B$&&$CD$\\
$CD$&&$B$&$C$&$D$\\
$D$&&$B$&$C$&$DD$\\
$DD$&&$B$&$C$&$DD$\\\hline
\end{tabular}}

From this automaton the generating function for the simple permutations, $\pi$, of $\A$ with $\pi(2) \neq 1$ can easily be computed using the transfer matrix approach:
\[
\frac{x^4}{(1-3x)(1+x)}.
\]
Thus, the generating function for the simples in $\A$ is twice this. In the enumeration of the entire class we will not actually use this generating function, because the allowed inflations of points of these simple permutations depend on which letter they were encoded by. For those purposes we need a more detailed form of the generating function of the language accepted by the automaton which treats the letters of the alphabet as distinct commuting variables:
\begin{equation}
\label{EQ_Automaton_GF}
 {\frac {{d}^{4} \left( 1+b \right) }{1 - 2\,abc-ac-ab-bd-bc-bcd-cd-d}}.
\end{equation}

\subsection{Completing the enumeration}

To complete the enumeration, we need to establish how points of the simple permutations can be inflated. Appealing to Figure~\ref{FIG_Encoding}, we first observe that every point which is encoded by $\b$ can only be inflated to $\Av(21)$, i.e.~an increasing sequence, else a 4231 pattern would be created by using suitable points encoded by $\d$.

Next, in order to avoid 4231, every point encoded by $\a$ or $\c$ can be inflated by a permutation that avoids either 231, or 312, depending on whether it lies above and left, or below and right (respectively) of a point encoded by $\d$. In either case, this is the only restriction, since all the other basis elements of the class $\A$ contain both 231 and 312.

By a similar argument, points encoded by $\d$ which lie above the crenellation can be inflated by any permutation from $\Av(312)$, and those below by any permutation in $\Av(231)$. Both of these classes are enumerated by the Catalan numbers, whose generating function (excluding the empty permutation) is
\[
f_{cat} = \frac{1-2x-\sqrt{1-4x}}{2x}.
\]

Consequently, by substituting $f_{cat}$ in place of every letter $\a$, $\c$ or $\d$, and $x/(1-x)$ in place of every $\b$ in expression (\ref{EQ_Automaton_GF}) then doubling (to take account of the symmetry), the generating function for the inflations of the simple permutations of length four or more in $\A$ is
\[
f^s_\A = \frac{(1-2x-\sqrt{1-4x})^4}{8x^2(x^2+3x-1+\sqrt{1-4x}-x\sqrt{1-4x})}.
\]

We now turn our attention to the enumeration of the skew decomposable permutations in $\A$. If $\pi\in\A$ is skew decomposable, then $\pi=21[\pi_1,\pi_2]$ where to ensure a unique representation we take $\pi_1$ to be skew indecomposable. Because of the basis element 4231, $\pi_1\in\Av(312)$, and $\pi_2\in\Av(231)$ (and moreover these necessary conditions are sufficient to guarantee $\pi \in \A$, as the remaining basis elements of $\A$ are skew indecomposable and involve both 312 and 231). The generating function for the skew indecomposables in $\Av(312)$ is given by $f_{cat}(1-x)$, and so the generating function for the skew decomposable permutations of $\A$ is
\[
f^{\ominus}_\A = f^2_{cat}(1-x).
\]

Finally, we observe (by inspecting the basis elements) that for any $\alpha, \beta \in \A$, $12[\alpha, \beta] \in \A$. In particular any member of $\A$ is a sum of a unique sum indecomposable member of $\A$ and another (possibly empty) member of $\A$. Since the sum indecomposables are: 1, inflations of 21, and inflations of simple permutations of length at least four we obtain:
\[
f_{\A} = (x + f^{\ominus}_\A + f^{s}_{\A})(1 + f_{\A})
\]
where $f_{\A}$ is the generating function for the class $\A$. Using the relations above and solving for $f_{\A}$, we have
\[
f_\A = \frac{1-3x-2x^2-(1-x-2x^2)\sqrt{1-4x}}{1-3x-(1-x+2x^2)\sqrt{1-4x}}.
\]

The sequence of coefficients for this generating function begins 1, 2, 6, 23, 101, 477, 2343, 11762, 59786, 306132, \dots  \OEIS{A213090}

\section{Conclusions}

As noted in the introduction, the structure of the argument used to eventually determine the generating function of $\A$ owes a great deal to a sequence of results concerning simple permutations and ``grid classes''. The interested reader may wish to consult \cite{albert:geometric-grid-:, brignall:a-survey-of-sim:} and references therein. The papers \cite{albert:the-enumeration:2143:4231, albert:counting-1324-4:} contain other examples of the general style of argument that we have used here. It is worth noting however that the ``grid'' illustrated in Figure \ref{FIG_Crenellation} is an infinite one, so the results mentioned above cannot be applied directly to this problem. However, it does seem that they could still be profitably applied for instance to other permutation classes such as those mentioned in \cite{ulfarsson:which-schubert:}.

The reversals of the permutations considered in this paper also appear (conjecturally) in yet another context related to $q$-enumerations of certain classes of matrices. This connection is described in \cite{2012arXiv1203.5804K}.

Exploration of the structure of simple permutations in $\A$ given their extreme pattern was greatly facilitated by the use of \emph{PermLab} (\url{http://www.cs.otago.ac.nz/PermLab}). The derivation of the generating functions for $\A$ and the associated automaton was carried out in \emph{Maple}. The authors would like once again to thank Alexander Woo for bringing the problem of enumerating $\A$ to their attention at this time, and also Masaki Ikeda and Sara Billey who noted typographical errors in the arXiv version. Finally, we would like to thank an attentive and generous referee. It seems that the first author's expressed view at FPSAC 2008 that, given the conjectured generating function its confirmation was a ``moderate exercise'' proved to be somewhat premature.

\bibliographystyle{model1-num-names}
\bibliography{refs}

\begin{thebibliography}{10}
\expandafter\ifx\csname natexlab\endcsname\relax\def\natexlab#1{#1}\fi
\providecommand{\url}[1]{\texttt{#1}}
\providecommand{\href}[2]{#2}
\providecommand{\path}[1]{#1}
\providecommand{\DOIprefix}{doi:}
\providecommand{\ArXivprefix}{arXiv:}
\providecommand{\URLprefix}{URL: }
\providecommand{\Pubmedprefix}{pmid:}
\providecommand{\doi}[1]{\href{http://dx.doi.org/#1}{\path{#1}}}
\providecommand{\Pubmed}[1]{\href{pmid:#1}{\path{#1}}}
\providecommand{\bibinfo}[2]{#2}
\ifx\xfnm\relax \def\xfnm[#1]{\unskip,\space#1}\fi
\bibitem[{Lakshmibai and Sandhya(1990)}]{lakshmibai:criterion-for-smoothness:}
\bibinfo{author}{V.~Lakshmibai}, \bibinfo{author}{B.~Sandhya},
\newblock \bibinfo{title}{Criterion for smoothness of {S}chubert varieties in
  {${\rm Sl}(n)/B$}},
\newblock \bibinfo{journal}{Proc. Indian Acad. Sci. Math. Sci.}
  \bibinfo{volume}{100} (\bibinfo{year}{1990}) \bibinfo{pages}{45--52}.
\bibitem[{Gasharov and Reiner(2002)}]{gasharov:cohomology:}
\bibinfo{author}{V.~Gasharov}, \bibinfo{author}{V.~Reiner},
\newblock \bibinfo{title}{Cohomology of smooth {S}chubert varieties in partial
  flag manifolds},
\newblock \bibinfo{journal}{J. London Math. Soc. (2)} \bibinfo{volume}{66}
  (\bibinfo{year}{2002}) \bibinfo{pages}{550--562}.
\bibitem[{Sj{\"o}strand(2007)}]{sjostrand:bruhat:}
\bibinfo{author}{J.~Sj{\"o}strand},
\newblock \bibinfo{title}{Bruhat intervals as rooks on skew {F}errers boards},
\newblock \bibinfo{journal}{J. Combin. Theory Ser. A} \bibinfo{volume}{114}
  (\bibinfo{year}{2007}) \bibinfo{pages}{1182--1198}.
\bibitem[{Hultman et~al.(2009)Hultman, Linusson, Shareshian, and
  Sj{\"o}strand}]{hultman:from-bruhat:}
\bibinfo{author}{A.~Hultman}, \bibinfo{author}{S.~Linusson},
  \bibinfo{author}{J.~Shareshian}, \bibinfo{author}{J.~Sj{\"o}strand},
\newblock \bibinfo{title}{From {B}ruhat intervals to intersection lattices and
  a conjecture of {P}ostnikov},
\newblock \bibinfo{journal}{J. Combin. Theory Ser. A} \bibinfo{volume}{116}
  (\bibinfo{year}{2009}) \bibinfo{pages}{564--580}.
\bibitem[{Albert et~al.(2013)Albert, Atkinson, Bouvel, Ru{\v{s}}kuc, and
  Vatter}]{albert:geometric-grid-:}
\bibinfo{author}{M.~H. Albert}, \bibinfo{author}{M.~D. Atkinson},
  \bibinfo{author}{M.~Bouvel}, \bibinfo{author}{N.~Ru{\v{s}}kuc},
  \bibinfo{author}{V.~Vatter},
\newblock \bibinfo{title}{Geometric grid classes of permutations},
\newblock \bibinfo{journal}{Trans. Amer. Math. Soc.} \bibinfo{volume}{365}
  (\bibinfo{year}{2013}) \bibinfo{pages}{5859--5881}.
\bibitem[{Brignall(2010)}]{brignall:a-survey-of-sim:}
\bibinfo{author}{R.~Brignall},
\newblock \bibinfo{title}{A survey of simple permutations},
\newblock in: \bibinfo{editor}{S.~Linton}, \bibinfo{editor}{N.~Ru{\v{s}}kuc},
  \bibinfo{editor}{V.~Vatter} (Eds.), \bibinfo{booktitle}{Permutation
  Patterns}, volume \bibinfo{volume}{376} of \textit{\bibinfo{series}{London
  Mathematical Society Lecture Note Series}}, \bibinfo{publisher}{Cambridge
  University Press}, \bibinfo{year}{2010}, pp. \bibinfo{pages}{41--65}.
\bibitem[{Albert et~al.(2011)Albert, Atkinson, and
  Brignall}]{albert:the-enumeration:2143:4231}
\bibinfo{author}{M.~H. Albert}, \bibinfo{author}{M.~D. Atkinson},
  \bibinfo{author}{R.~Brignall},
\newblock \bibinfo{title}{The enumeration of permutations avoiding $2143$ and
  $4231$},
\newblock \bibinfo{journal}{Pure Mathematics and Applications}
  \bibinfo{volume}{22} (\bibinfo{year}{2011}) \bibinfo{pages}{87--98}.
\bibitem[{Albert et~al.(2009)Albert, Atkinson, and
  Vatter}]{albert:counting-1324-4:}
\bibinfo{author}{M.~H. Albert}, \bibinfo{author}{M.~D. Atkinson},
  \bibinfo{author}{V.~Vatter},
\newblock \bibinfo{title}{Counting $1324$, $4231$-avoiding permutations},
\newblock \bibinfo{journal}{Electron. J. Combin.} \bibinfo{volume}{16}
  (\bibinfo{year}{2009}) \bibinfo{pages}{Research Paper 136, 9}.
\bibitem[{{{\'U}lfarsson} and {Woo}(2011)}]{ulfarsson:which-schubert:}
\bibinfo{author}{H.~{{\'U}lfarsson}}, \bibinfo{author}{A.~{Woo}},
\newblock \bibinfo{title}{{Which Schubert varieties are local complete
  intersections?}},
\newblock \bibinfo{journal}{ArXiv e-prints}  (\bibinfo{year}{2011}).
\bibitem[{{Klein} et~al.(2012){Klein}, {Brewster Lewis}, and
  {Morales}}]{2012arXiv1203.5804K}
\bibinfo{author}{A.~J. {Klein}}, \bibinfo{author}{J.~{Brewster Lewis}},
  \bibinfo{author}{A.~H. {Morales}},
\newblock \bibinfo{title}{{Counting matrices over finite fields with support on
  skew Young diagrams and complements of Rothe diagrams}},
\newblock \bibinfo{journal}{ArXiv e-prints}  (\bibinfo{year}{2012}).

\end{thebibliography}

\end{document}